\documentclass[12pt]{article}
\usepackage{amsmath,amssymb,amsthm}
\usepackage{lmodern}
\usepackage{fbb}
\usepackage{tikz}
\usetikzlibrary{arrows.meta, positioning, decorations.markings, calc, shapes.misc}
\usepackage[libertine]{newtxmath}
\usepackage[margin=2cm,a4paper]{geometry}
\usepackage{authblk}
\usepackage{enumitem}
\usepackage{graphicx} \usepackage{lineno}
\usepackage{hyperref}
\hypersetup{
    colorlinks=true,
    citecolor=blue,
    linkcolor=blue,
    filecolor=blue,      
    urlcolor=blue,
}
\usepackage{keytheorems}
\usepackage{zref-clever}
\zcsetup{cap=true}
\newkeytheorem{theorem}[name=Theorem,parent=section,refname={Theorem,Theorems}, Refname={Theorem,Theorems}]
\newkeytheorem{lemma,proposition,corollary,conjecture,problem}[sibling=theorem]

\newcommand\rank{\operatorname{rank}}
\newcommand\abs[1]{\left\lvert #1\right\rvert}

\newcommand\Out{\operatorname{Out}}
\newcommand\In{\operatorname{In}}

\begin{document}
\title{Polynomial-size encoding of all cuts of small value in integer-valued symmetric submodular functions}
\author{Sang-il~Oum\thanks{Supported by the Institute for Basic Science (IBS-R029-C1)}}
\affil{Discrete Mathematics Group, Institute for Basic Science (IBS), Daejeon,~South~Korea}
\affil{Department of Mathematical Sciences, KAIST, Daejeon, South~Korea}
\author{Marek Soko\l{}owski}
\affil{Max Planck Institute for Informatics, Saarland Informatics Campus, Saarbr\"ucken, Germany}
\affil[ ]{\small \textit{Email addresses:} 
\texttt{sangil@ibs.re.kr}, \texttt{msokolow@mpi-inf.mpg.de}}
\date{March 23, 2026}

\maketitle
\begin{abstract}
We study connectivity functions, that is, integer-valued symmetric submodular functions on a finite ground set attaining $0$ on the empty set.
For a connectivity function~$f$ on an $n$-element set $V$ and an integer $k\ge 0$, we show that the family of all sets $X\subseteq V$ with $f(X)=k$ admits a polynomial-size representation: it can be described by a list of at most $O(n^{4k})$ items, each consisting of a set to be included, another set to be excluded, and a partition of remaining elements, such that the union of some members of the partition and the set to be included are precisely all sets $X$ with $f(X)=k$. 
We also give an algorithm that constructs this representation in time $O(n^{2k+7}\gamma+n^{2k+8}+n^{4k+2})$, where $\gamma$ is the oracle time to evaluate~$f$. 
This generalizes the low rank structure theorem of Boja\'nczyk, Pilipczuk, Przybyszewski, Soko\l{}owski, and Stamoulis~[Low rank MSO, arXiv, 2025] on cut-rank functions on graphs to general connectivity functions.
As an application, for fixed $k$, we obtain a polynomial-time algorithm for finding a set $A$ with $f(A)=k$ and a prescribed cardinality constraint on $A$.
\end{abstract}

\section{Introduction}\label{sec:intro}

A \emph{connectivity function} on a finite set~$V$ 
is a function $f:2^V\to\mathbb Z$ satisfying the following three conditions.
\begin{enumerate}[label=\rm(\alph*)]
    \item (symmetric) $f(X)=f(V-X)$ for all $X\subseteq V$.
    \item (submodular) $f(X)+f(Y)\ge f(X\cup Y)+f(X\cap Y)$ for all $X,Y\subseteq V$.
    \item $f(\emptyset)=0$.
\end{enumerate}
Here are several examples of connectivity functions. See~\cite{Fujishige1983,Queyranne1998} for more instances of symmetric submodular functions.
\begin{itemize}
    \item Vertex cut: Given a graph $G$ and a set $X$ of edges, let $\nu(X)$ be the number of vertices incident with both an edge in $X$ and an edge in $E(G)-X$.
    \item Edge cut: Given a graph $G$ and a set $X$ of vertices, let $\eta(X)$ be the number of edges having one end in $X$ and the other end in $V(G)-X$.
    \item Cut-rank: Given a graph and a set $X$ of vertices, let $\rho(X)$ be the rank of the $X\times (V(G)-X)$ $0$-$1$ matrix over the binary field where the $ij$ entry of the matrix is~$1$ if and only if the vertex~$i$ is adjacent to the vertex~$j$. See \cite{OS2004}.
    \item Matroid: Given a matroid $M$ with the rank function $r$ and a set $X$ of elements, let $\lambda(X)=r(X)+r(E(M)-X)-r(E(M))$.
\end{itemize}

We prove the following theorem. 
This is motivated by the low rank structure theorem of Boja\'nczyk, Pilipczuk, Przybyszewski, Soko\l{}owski, and Stamoulis~\cite[Theorem 6.1]{BPPSS2025}, 
which covers the special case of $f$ being the cut-rank function of a graph.
\getkeytheorem{main}
As a corollary, we obtain for each fixed $k$ and fixed set $W$, a polynomial-time algorithm to find a set $A$ with $f(A\cap W)=k$ and $\abs{A}\in T$, if it exists for an input connectivity function $f$ on~$V$ given by an oracle and a set $W$.
If $T=\{\lceil \abs{V}/2\rceil\}$ and $W=V$, then this is the \textsc{submodular minimum bisection} problem parameterized by the function value for connectivity functions.

\getkeytheorem{bisection}

This paper is organized as follows. 
\zcref{sec:prelim} recalls definitions of interpolations of a connectivity function and discuss several properties.
\zcref{sec:blocking} defines the blocking digraphs.
\zcref{sec:matching} discusses our key lemma, saying that the blocking digraph cannot have a structure called a skew matching.
\zcref{sec:closed} describes all closed sets in a directed acyclic graph, when there is no large skew matching.
\zcref{sec:proof} finally presents the proof of the main theorem.
\zcref{sec:cor} gives the proof of the corollary.

\section{Preliminaries}\label{sec:prelim}

Let us start this section by reviewing some terms on directed graphs, also called digraphs. 
A vertex of a digraph is called a \emph{sink} if it has no out-going  arcs 
and it is called a \emph{source} if its has no incoming arcs.
For a set~$X$ of vertices, we write $\Out(X)$ to denote the set of all vertices~$y$ such that there is a directed path from a vertex in~$X$ to $y$.
Similarly, we write $\In(X)$ to denote the set of all vertices~$y$ such that there is a directed path from~$y$ to a vertex in~$X$.
We write $D[X]$ to denote the subgraph of~$D$ induced by~$X$, that is the directed graph obtained by deleting all vertices outside of $X$ and all arcs incident with a vertex out of $X$.
A set~$X$ of vertices of a directed graph is \emph{independent} if $D[X]$ has no arcs.

Now let us recall the definition of an interpolation of a connectivity function, introduced in~\cite{OS2004}.
We write $3^V$ to denote a set of disjoint subsets $\{ (X,Y): X, Y\subseteq V,~X\cap Y=\emptyset\}$.
Let $f:2^V\to\mathbb Z$ be a connectivity function. 
We say that $f^*:3^V\to\mathbb Z$ is an \emph{interpolation} of~$f$ 
if it satisfies the following four conditions. 
\begin{enumerate}[label=\rm(\roman*)]
    \item $f^*(X,V-X)=f(X)$ for all $X\subseteq V$.
    \item (monotone) If $C\cap D=\emptyset$, $A\subseteq C$, and $B\subseteq D$, then $f^*(A,B)\le f^*(C,D)$.
    \item (submodular) $f^*(A,B)+f^*(C,D)\ge f^*(A\cap C,B\cup D)+f^*(A\cup C,B\cap D)$ for all $(A,B),(C,D)\in 3^V$.
    \item $f^*(\emptyset,\emptyset)=f(\emptyset)$.
\end{enumerate}

Every connectivity function admits a canonical interpolation $f_{\min}(A,B)=\min_{A\subseteq X\subseteq V-B} f(X)$ \cite[Proposition 4.2]{OS2004}.
For a graph $G$, if we say $\rho_G^*(S,T)=\rank A(G)[S,T]$ where $A(G)$ is the adjacency matrix of~$G$ over the binary field, then $\rho_G^*$ is an interpolation of the cut-rank function of the graph, see \cite{OS2004}.
The following result implies that $f_{\min}$ is polynomial-time computable as long as $f$ is:

\begin{theorem}[note={Orlin~\cite{Orlin2009}},label=thm:orlin]
    Let $n$ be a positive integer.
    Let $f:2^V\to\mathbb Z$ be a submodular function on an $n$-element set~$V$.
    Let us assume that $\gamma$ is the time to compute $f(X)$ for any subset $X$ of~$V$.
    Then we can find $X\subseteq V$ minimizing $f(X)$ in time $O(\gamma n^5 + n^6)$. 
\end{theorem}

Our aim is to describe subsets $X$ of $V$ with $f(X)=k$.
The following two lemmas state that
if $f(X)=k$, then there will be a pair $(A,B)$ of small sets such that 
$A\subseteq X$, $B\subseteq V-X$, and $f^*(A,B)=k$.

\begin{lemma}\label{lem:base}
    Let $r:2^V\to \mathbb Z$ be a submodular function such that $r(X)\le r(Y)$  for all $X\subseteq Y\subseteq V$ and $r(\emptyset)=0$. 
    Then there exists a set $X$ with $r(X)=r(V)$ and $\abs{X}\le r(V)$.
    Moreover, if we can compute $r(Z)$ in time $\gamma$ for any subset $Z$ of $V$, then we can find such a set $X$ in time $O(\gamma \abs{V})$.
\end{lemma}
\begin{proof}
    We proceed by induction on $\abs{V}$.
    This statement is trivially true if $r(V)=0$. Thus we may assume that $r(V)>0$. 
    Let $v\in V$.
    If $r(V-\{v\})=r(V)$, then by the induction hypothesis applied to $V-\{v\}$,
    there is a subset $X\subseteq V-\{v\}$ such that
    $r(X)=r(V-\{v\})=r(V)$ and $\abs{X}\le r(V)$, as desired.
    Therefore we may assume that $r(V-\{v\})<r(V)$.
    By the induction hypothesis applied to $V-\{v\}$, there is a subset $X$ of $V-\{v\}$ such that $r(X)=r(V-\{v\})$ and $\abs{X}\le r(V-\{v\})$.

    For every $w\in V-(X\cup\{v\})$, we have
    $r(X)\le r(X\cup\{w\})\le r(V-\{v\})=r(X)$,
    and therefore $r(X\cup\{w\})=r(X)$.
    Let $Y$ be a maximal subset of $V$ such that $X\cup\{v\}\subseteq Y$ and $r(Y)=r(X\cup\{v\})$. 
    If $Y\neq V$, then for $w\in V-Y$, we have 
    $r(X\cup \{w\})+r(Y)\ge r(Y\cup \{w\})+r(X)$
    and therefore $r(Y\cup \{w\})= r(Y)$, contradicting 
    the assumption that $Y$ is maximal. 
    Therefore, $Y=V$. This implies that $r(X\cup\{v\})=r(V)$.
    Note that $\abs{X}+1 \le r(V-\{v\})+1 \le r(V)$.
    Thus $X\cup \{v\}$ is a desired set.
\end{proof}
\begin{lemma}\label{lem:findbase}
    Let $f$ be a connectivity function on a finite set~$V$.
    Let $f^*$ be an interpolation of~$f$.
    Let $X$ be a subset of~$V$.
    If $f(X)=k$, then 
    there exist $A\subseteq X$, $B\subseteq V-X$ such that 
    $\max(\abs{A},\abs{B})\le k$
    and $f^*(A,B)=k$.
    Moreover, if we can compute $f^*(S,T)$ in time $\gamma$ for any pair $(S,T)$ of disjoint subsets of $V$, then we can find such a pair $(A,B)$ in time $O(\gamma \abs{V})$.
\end{lemma}
\begin{proof}
    We may assume that $k>0$.
    By \zcref{lem:base}, there is a subset $A$ of $X$ such that 
    $f^*(A,V-X)=k$ and $\abs{A}\le k$.
    Again by \zcref{lem:base}, there is a subset $B$ of $V-X$ such that 
    $f^*(A,B)=k$ and $\abs{B}\le k$. 
\end{proof}

\section{Blocking digraphs}\label{sec:blocking}
Let $f$ be a connectivity function on a finite set $V$.
Let $f^*$ be an interpolation of $f$.
For two disjoint subsets $S$ and $T$ of~$V$, 
let us define an auxiliary digraph $D_{S,T}$ on $(V-(S\cup T))\dot\cup \{S^\circ,T^\circ\} $ as follows.

\begin{enumerate}[label=\rm (\alph*)]
    \item For $x\in V-(S\cup T)$, if $f^*(S,T\cup \{x\})>f^*(S,T)$, then $(S^\circ,x)\in A(D_{S,T})$.
    \item For $x\in V-(S\cup T)$, if $f^*(S\cup\{x\},T)>f^*(S,T)$, then $(x,T^\circ)\in A(D_{S,T})$.
    \item For distinct $x,y\in V-(S\cup T)$, if $f^*(S\cup\{x\},T\cup \{y\})>f^*(S,T)$, then $(x,y)\in A(D_{S,T})$.
\end{enumerate}
Let us call this digraph the \emph{blocking digraph} for $(S,T)$ with respect to $f^*$.
This is a digraph appearing in the proof for properties on  \emph{blocking sequences} introduced by Geelen~\cite[Chapter 5]{Geelen1995}.
A set~$X$ of vertices of a digraph~$D$ is a \emph{closed set}
if $D$ has no arcs leaving~$X$.

\begin{lemma}\label{lem:digraph}
    Let $f$ be a connectivity function on a finite set $V$
    and $f^*$ be an interpolation of $f$.
Let $A$, $B$, $S$, $T$ be pairwise disjoint subsets of $V$.
    Let $D$ be the blocking digraph for $(S,T)$ with respect to $f^*$.
Then $f^*(S\cup A,T\cup B)=f^*(S,T)$ if and only if 
    $D$ has no arcs from $A\cup \{S^\circ\}$ to $B\cup \{T^\circ\}$.
\end{lemma}
\begin{proof}
    Let $k=f^*(S,T)$.

    The forward direction is easy. 
    Suppose that $D$ has an arc  from $A\cup\{S^\circ\}$ to $B\cup\{T^\circ\}$.
    Then there is a subset $A'$ of $A$ and a subset $B'$ of $B$ such that $\abs{A'}\le 1$, $\abs{B'}\le 1$, and $f^*(S\cup A',T\cup B')>k$. 
    By the monotonicity of $f^*$, 
    $f^*(S\cup A,T\cup B)\ge f^*(S\cup A',T\cup B')>k$, contradicting the assumption.

    The backward direction can be found in the proof of \cite[Proposition 3.1]{Oum2006} or any proofs for applying blocking sequences. We include the proof for completeness.
    We proceed by induction on~$\abs{A}+\abs{B}$.
    If $\abs{A}+\abs{B}=0$, then $A=B=\emptyset$ and the statement is trivial.

    If $\abs{A}\ge 2$, let $v\in A$.
    By the induction hypothesis, we have
    $f^*(S\cup (A-\{v\}),T\cup B)=k$,
    $f^*(S\cup \{v\},T\cup B)=k$, and
    $f^*(S,T\cup B)=k$.
    By submodularity,
    \[
    f^*(S\cup (A-\{v\}),T\cup B)+f^*(S\cup \{v\},T\cup B)
    \ge f^*(S\cup A,T\cup B)+f^*(S,T\cup B).
    \]
    Hence $f^*(S\cup A,T\cup B)\le k$.
    Since $f^*$ is monotone, we have $f^*(S\cup A,T\cup B)\ge f^*(S,T)=k$.
    Therefore $f^*(S\cup A,T\cup B)=k$.
    This completes this case.

    If $\abs{B}\ge 2$, 
    then by the same argument, we deduce that $f^*(S\cup A,T\cup B)=k$.

    Now $\abs{A}\le 1$ and $\abs{B}\le 1$.
    If $A=\emptyset$ and $B=\{w\}$, then there is no arc from $S^\circ$ to $w$, and therefore
    $f^*(S,T\cup \{w\})=k$.
    If $A=\{v\}$ and $B=\emptyset$, then there is no arc from $v$ to $T^\circ$, and therefore
    $f^*(S\cup \{v\},T)=k$.
    If $A=\{v\}$ and $B=\{w\}$, then there is no arc from $v$ to $w$, and so 
    $f^*(S\cup \{v\},T\cup \{w\})=f^*(S,T)$.
    Hence $f^*(S\cup A,T\cup B)=k$ in all remaining cases.
\end{proof}

Now our aim is to describe all subsets $X$ of $V-(S\cup T)$ with $f(S\cup X)=k$ for every pair $(S,T)$ of disjoint subsets with $\max(\abs{S},\abs{T})\le f^*(S,T)=k$.
Let $D$ be the blocking digraph for $(S,T)$ with respect to~$f^*$.
Let $X$ be a subset of $V(D)$ with $S^\circ\in X$ and $T^\circ\notin X$.
If $D$ has no arcs leaving $X$, then for every strongly connected component~$C$ of $D$, either $V(C)\subseteq X$ or $V(C)\cap X=\emptyset$.
Our aim is to describe all sets $X$ of vertices of $D$ containing $S^\circ$ and not containing $T^\circ$ such that there are no arcs leaving $X$.
Then $X$ cannot split a strongly connected component 
and therefore we can reduce the problem to a directed acyclic graph obtained by contracting each strongly connected component.
Furthermore, for the purposes of the upcoming arguments (and the resulting algorithm), we can remove from $D$ any vertex $v$ that either can be reached from $S^\circ$ by a directed path 
or can reach~$T^\circ$ by a directed path: in the former case, this vertex has to be in $X$, and in the latter, it cannot be in~$X$.
Then what is left for us is to characterize all sets $X$ of vertices of a \emph{directed acyclic graph} having no arcs leaving $X$, called a \emph{closed set}.

\section{Forbidden skew matching}\label{sec:matching}

We say that an interpolation $f^*$ is \emph{symmetric} if $f^*(X,Y)=f^*(Y,X)$ for all disjoint subsets $X$, $Y$.
Since $f_{\min}$ is symmetric, every connectivity function admits a symmetric interpolation.

\begin{figure}[t]
\centering
\begin{tikzpicture}[
  >=Stealth,
  vtx/.style={circle,draw,inner sep=1.2pt,minimum size=3pt,fill},
  match/.style={->,line width=1.0pt},
  allow/.style={->,dotted,line width=0.9pt},
  x=1.55cm,y=1.35cm
]

\def\L{5}

\draw[draw,rounded corners=0pt,fill,color=gray!10] (3,1) ellipse [x radius=2.5, y radius=0.15];

\draw[draw,rounded corners=0pt,fill,color=gray!10] (3,0) ellipse [x radius=2.5, y radius=0.15];

\foreach \i in {1,...,\L}{
  \node[vtx,label=$a_{\i}$] (a\i) at (\i,1) {};
  \node[vtx,label=below:$b_{\i}$] (b\i) at (\i,0) {};
}

\foreach \i in {1,...,\L}{
  \draw[match] (a\i) -- (b\i);
}

\foreach \i in {1,...,\L}{
  \foreach \j in {1,...,\L}{
    \ifnum\j>\i
      \draw[allow] (a\j) to (b\i);
    \fi
  }
}

\node[align=left,anchor=west] at (\L+0.75,1.05)
{$\bullet$ solid: matching arcs $(a_i,b_i)$\\
$\bullet$ dotted: allowable arcs $a_j\to b_i$ for $j>i$};

\end{tikzpicture}
\caption{A skew matching of size $5$.  The solid arcs form the matching
$(a_i,b_i)$, while dotted arcs indicate other allowable arcs.
No arcs $a_i\to b_j$ with $i<j$ and no arcs $b_i\to a_j$ are present.
Arcs between $a_i$ and $a_j$ or between $b_i$ and $b_j$ are allowed.}
\label{fig:matching}
\end{figure}
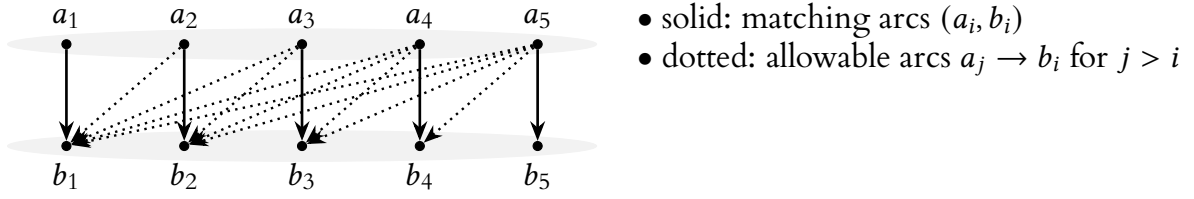

We say that a set of arcs $(a_1,b_1)$, $(a_2,b_2)$, $\ldots$, $(a_\ell,b_\ell)$ 
in a directed graph $D$ is a \emph{skew matching}
if $a_1,a_2,\ldots,a_\ell,b_1,b_2,\ldots,b_\ell$ are pairwise disjoint vertices of $D$ such that 
\begin{enumerate}[label=\rm(\roman*)]
    \item $(a_i,b_i)\in E(D)$ for all $i\in[\ell]$,
    \item $(a_i,b_j)\notin E(D)$ for all $1\le i<j\le \ell$, 
    \item $(b_i,a_j)\notin E(D)$ for all $i,j\in [\ell]$.
\end{enumerate} 
See \zcref{fig:matching}.
We say $\ell$ is the size of the skew matching.
Our next lemma shows that 
if $f^*$ is a symmetric interpolation of a connectivity function~$f$ and $D$ is the blocking digraph for $(S,T)$ with respect to $f^*$, 
then $D-(\Out(S^\circ)\cup \In(T^\circ))$  
does not admit a skew matching of size $2f^*(S,T)+1$.

\begin{lemma}\label{lem:nomatching}
    Let $f^*$ be a symmetric interpolation of a connectivity function~$f$ on a finite set~$V$. 
    Let $S$, $T$ be disjoint subsets of $V$.
    If there exist distinct elements 
    $a_1,a_2,\ldots,a_\ell,b_1,b_2,\ldots,b_\ell$ in $V - (S \cup T)$ 
    such that 
    \begin{enumerate}[label=\rm(\roman*)]
        \item \label{itm:aibi} $f^*(S\cup \{a_i\},T\cup\{b_i\})>f^*(S,T)$  for all $i\in\{1,2,\ldots,\ell\}$, 
        \item \label{itm:aibj} $f^*(S\cup \{a_i\},T\cup\{b_j\})=f^*(S,T)$ for all $1\le i<j\le \ell$, 
        \item \label{itm:biaj} $f^*(S\cup \{b_i\},T\cup\{a_j\})=f^*(S,T)$ for all $i,j\in\{1,2,\ldots,\ell\}$, and 
        \item \label{itm:nost} $f^*(S\cup\{a_i\},T)=f^*(S,T\cup\{b_i\})=f^*(S,T)$ for all $i\in \{1,2,\ldots,\ell\}$,    
    \end{enumerate}
    then $\ell\le 2f^*(S,T)$.
\end{lemma}
\begin{proof}
    Let $k=f^*(S,T)$. 
    For $i\in\{1,2,\ldots,\ell\}$, let $A_i=\{a_1,a_2,\ldots,a_i\}$ and $B_i=\{b_1,b_2,\ldots,b_i\}$. 
    By the assumption, $(S\cup T)\cap (A_\ell \cup B_\ell)=\emptyset$.
    First we claim that $f^*(S\cup A_i,T\cup B_i)\ge k+i$ for all $i\in\{1,2,\ldots,\ell\}$.
    This is obviously true if $i=1$.
    If $i>1$, then  by submodularity and monotonicity, we have 
    \begin{align*}
        \lefteqn{f^*(S\cup A_i,T\cup B_i)
        + f^*(S\cup A_{i-1},T\cup\{b_i\})}\\
        &\ge
        f^*(S\cup A_{i-1},
        T\cup B_i)
        +f^*(S\cup A_i,
        T\cup\{b_i\})\\
        &\ge 
        f^*(S\cup A_{i-1},
        T\cup B_{i-1})
        + f^*(S\cup\{a_i\},
        T\cup\{b_i\})\\
        \intertext{and by the induction hypothesis and \ref{itm:aibi},}
        &\ge (k+i-1)+(k+1).
    \end{align*}
    Note that by \ref{itm:aibj}, \ref{itm:nost}, and \zcref{lem:digraph},
    $f^*(S\cup A_{i-1},T\cup \{b_i\})=k$.
    Therefore we deduce that $f^*(S\cup A_i,T\cup B_i)\ge k+i$.

    By the claim, we deduce that 
    $f^*(S\cup A_\ell,
    T\cup B_\ell)\ge k+\ell$.
    Since $f^*$ is submodular and $f^*(\emptyset,\emptyset)=f(\emptyset)=0$, 
    we have 
    \begin{align*} 
    f^*(S\cup A_\ell,B_\ell)&
    \ge 
    f^*(S\cup A_\ell,T\cup B_\ell) 
    +f^*(S\cup A_\ell,\emptyset)
    - f^*(S\cup A_\ell, T)\\
    &\ge
    f^*(S\cup A_\ell,T\cup B_\ell) 
    - f^*(S\cup A_\ell, T), \\ 
    f^*(A_\ell,B_\ell)&\ge
    f^*(S\cup A_\ell,B_\ell) 
    +f^*(\emptyset,B_\ell)- f^*(S,B_\ell)
    \ge 
    f^*(S\cup A_\ell,B_\ell) - f^*(S,B_\ell).
    \end{align*}
    By \zcref{lem:digraph} and \ref{itm:nost},  
    $f^*(S\cup A_\ell,T)=k$ and 
    $f^*(S,B_\ell)\le f^*(S,T\cup B_\ell)=k$.
Therefore we deduce that 
    \[ f^*(A_\ell,B_\ell)\ge k+\ell - 2k = \ell-k.\]
    Now, $f^*(B_\ell,A_\ell)\le f^*(S\cup B_\ell,T\cup A_\ell) =k $ by \zcref{lem:digraph} and \ref{itm:biaj}.
    Since $f^*$ is symmetric, this is equal to $f^*(A_\ell,B_\ell) \ge \ell -k$. 
    Thus, we deduce that $\ell\le 2k$.
\end{proof}

\section{Covering all closed sets}\label{sec:closed}
A \emph{sink} is a set of vertices having no out-neighbors.
For a set $S$ of vertices of a directed graph, let us write $N_D^-(S)$ to denote the set of all in-neighbors of vertices in $S$ outside of $S$
and write $N_D^+(S)$ to denote the set of all out-neighbors of vertices in $S$ outside of $S$. 
We omit the subscript $D$ to write $N^-(S)$ and $N^+(S)$, when the underlying digraph $D$ is clear from the context.

\begin{lemma}\label{lem:tridominating}
    Let $D$ be a directed graph and let $S_1$, $S_2$ be disjoint sets of vertices of~$D$
    such that 
    no two vertices of $S_1\cup S_2$ are joined by a directed path of length at most $3$. 
    Let $M$ be a minimal subset of $S_1\cup S_2$ such that\footnote{Note that such a~minimal $M$ exists since $M = S_1 \cup S_2$ satisfies \eqref{eq:marker}.} 
    \begin{equation}\label{eq:marker}
        \begin{aligned}
        N^+(M\cap S_1)&=N^+(S_1), & N^-(M\cap S_1)&=N^-(S_1),\\ 
        N^+(M\cap S_2)-N^+(S_1)&=N^+(S_2)-N^+(S_1),&
        N^-(M\cap S_2)-N^-(S_1)&=N^-(S_2)-N^-(S_1).
        \end{aligned}
    \end{equation}
    Then $D$ has a skew matching of size $\abs{M}$.
\end{lemma}
\begin{proof}
    Let us partition $M$ into four sets $M_1$, $M_2$, $M_3$, $M_4$ as follows.
    \begin{align*}
        M_1&= \{x\in M: N^+(M\cap S_1-\{x\})\neq N^+(S_1)\}, \\
        M_2 &= \{ x\in M-M_1: N^+(M\cap S_2-\{x\})-N^+(S_1)\neq N^+(S_2) - N^+(S_1) \}, \\ 
        M_3 &= \{x\in M-(M_1\cup M_2): N^-(M\cap S_2-\{x\})-N^-(S_1)\neq N^-(S_2) - N^-(S_1)\}, \\ 
        M_4&= \{x\in M-(M_1\cup M_2\cup M_3): N^-(M\cap S_1-\{x\})\neq N^-(S_1)\}.
    \end{align*}
    Since $M$ is minimal, we have $M=M_1\cup M_2\cup M_3\cup M_4$.

    Let $a_1,a_2,\ldots,a_p$ be the vertices of $M_1$. 
    For each $i\le p$, let $b_i$ be a vertex in $N^+(S_1)-N^+(M\cap S_1-\{a_i\})$. 
    
    Let $a_{p+1},a_{p+2},\ldots,a_{q}$ be the vertices of $M_2$.
    For each $p<i\le q$, let $b_i$ be a vertex in $N^+(S_2)-N^+(M\cap S_2-\{a_i\})-N^+(S_1)$.

    Let $b_{q+1},b_{q+2},\ldots,b_{r}$ be the vertices of $M_3$.
    For each $q<i\le r$, let $a_i$ be a vertex in $N^-(S_2)-N^-(M\cap S_2-\{b_i\})-N^-(S_1)$.
    
    Let $b_{r+1},b_{r+2},\ldots,b_s$ be the vertices of $M_4$.
    For each $r<i\le s$, let $a_i$ be a vertex in $N^-(S_1)-N^-(M\cap S_1-\{b_i\})$.

    We claim that $(a_1,b_1)$, $(a_2,b_2)$, $\ldots$, $(a_s,b_s)$ is a skew matching of $D$. Clearly, $(a_i,b_i)\in E(D)$ for all $i\le r$.

    If $i<j\le q$, then $(a_i,b_j)\notin E(D)$ by the choice of~$b_j$.
    If $q<i<j$, then $(a_i,b_j)\notin E(D)$ by the choice of~$a_i$.
    If $i\le q<j$, then $(a_i,b_j)\notin E(D)$ because otherwise $D$ has a directed path of length $1$ from $a_i\in M$ to $b_j\in M$.
    Therefore we deduce that $(a_i,b_j)\notin E(D)$ for all $i<j$.
    
    Finally, if $i,j\le q$, then $(b_i,a_j)\notin E(D)$ because otherwise $D$ has a directed path $a_ib_ia_j$ of length~$2$ from $a_i\in M$ to $a_j\in M$.
    If $i,j>q$, then $(b_i,a_j)\notin E(D)$ because otherwise $D$ has a directed path $b_ia_jb_j$ of length $2$ from $b_i\in M$ to $b_j\in M$.
    If $i\le q<j$, then $(b_i,a_j)\notin E(D)$ because otherwise $D$ has a directed path $a_ib_ia_jb_j$ of length $3$ from $a_i\in M$ to $b_j\in M$.
    If $j\le q<i$, then $(b_i,a_j)\notin E(D)$ because $D$ has no directed path of length $1$ from $b_i\in M$ to $a_j\in M$.
\end{proof}

\begin{lemma}\label{lem:encoding}
    Let $D$ be a directed acyclic graph with no skew matching of size $\ell+1$.
    Let $X$ be a closed set.
    Let $S_1$ be the set of all sources of $D[X]$.
    Let $S_2$ be the set of all sinks of $D-(X\cup \In(S_1))$.
    Let $M$ be a minimal subset of $S_1\cup S_2$ satisfying \eqref{eq:marker}.
    Then $\abs{M}\le \ell$ and 
    for $M_1=M\cap S_1$, $M_2=M\cap S_2$, and every subset $Y$ of $(S_1\cup S_2)-M$, 
    the set
    $Z=X\triangle Y$ satisfies the following. 
    \begin{enumerate}[label=\rm(\alph*)]
        \item \label{itm:indep0} $M_1\cup M_2\cup (V(D)-(\Out(M_1)\cup \In(M_1)\cup \In(M_2)))$ is an independent set.
        \item \label{itm:out0} $\Out(M_1)\subseteq Z$.
        \item \label{itm:in0} $Z\cap \In(M_1)= M_1$.
        \item \label{itm:side0} $Z\cap \In(M_2)=\emptyset$.
    \end{enumerate}
    
\end{lemma}
\begin{proof}
    By \zcref{lem:tridominating}, $\abs{M}\le \ell$.
Since $X$ is a closed set and $S_2\cap \In(S_1)=\emptyset$, we deduce that $S_1\cup S_2$ is independent.
    To prove \ref{itm:indep0}, let
    \[
    U:=\Out(M_1)\cup \In(M_1)\cup \In(M_2).
    \]
    We first show that
    \begin{equation}\label{eq:s1s2-rhs1}
        V(D)-U\subseteq S_1\cup S_2.
    \end{equation}
    Let $v\in V(D)-(S_1\cup S_2)$.
    Then $v\in \Out(S_1)-S_1$, $v\in \In(S_1)-S_1$, or $v\in \In(S_2)-S_2$.
    By \eqref{eq:marker}, $\Out(S_1)-S_1\subseteq \Out(M_1)$, 
    $\In(S_1)-S_1\subseteq \In(M_1)$,
    and $\In(S_2)-S_2\subseteq \In(M_2)\cup \In(M_1)$.
    Thus, we deduce that $v\in U$.
This proves \eqref{eq:s1s2-rhs1}.

    Next we show
    \begin{equation}\label{eq:s1s2-rhs2}
        S_1\cup S_2\subseteq M_1\cup M_2\cup (V(D)-U).
    \end{equation}
    It is enough to prove that
    $U$ contains no vertex in 
    $(S_1\cup S_2)-(M_1\cup M_2)$.
    Suppose that $v\in (S_1\cup S_2)\cap U$.
    Then $v\in \Out(M_1)\cup \In(M_1)\cup \In(M_2)$.
    If $v\in \Out(M_1)\cup \In(M_1)$, then $v\in X\cup \In(S_1)$ and therefore $v\notin S_2$, which implies that $v\in S_1$. So, we conclude that $v\in M_1$ in this case.
    If $v\in \In(M_2)-(\Out(M_1)\cup \In(M_1))$, then $v\notin S_1$, because otherwise $D$ has a directed path from $v\in S_1\subseteq X$ to some vertex in $M_2$ and $M_2\cap X=\emptyset$, contradicting the fact that $X$ is closed.
    So $v\in S_2$. Since $v\in \In(M_2)$, we conclude that $v\in M_2$. 
    This proves \eqref{eq:s1s2-rhs2}.

    By \eqref{eq:s1s2-rhs1} and \eqref{eq:s1s2-rhs2},
    \[
    S_1\cup S_2=M_1\cup M_2\cup (V(D)-(\Out(M_1)\cup \In(M_1)\cup \In(M_2))).
    \]
    Since $S_1\cup S_2$ is independent, \ref{itm:indep0} follows.

    As $X$ is a closed set and $M_1\subseteq X$, we deduce that $\Out(M_1)\subseteq X$. Since $Y\cap \Out(M_1)=\emptyset$, we obtain \ref{itm:out0}.

    Since $M_1$ is a set of sources of $D[X]$, $\In(M_1)\cap X=M_1$. Since $Y\cap \In(M_1)=\emptyset$, we deduce \ref{itm:in0}.

    Note that $X\cap \In(M_2)=\emptyset$, because otherwise $X$ contains a vertex in $M_2\subseteq S_2$, as $X$ is closed.
Then \ref{itm:side0} follows because $Y\cap \In(M_2)=\emptyset$.
    This completes the proof. 
\end{proof}

\begin{lemma}\label{lem:decoding}
    Let $D$ be a directed acyclic graph
    and let $M_1$, $M_2$ be disjoint sets of vertices
    and let $Z$ be a set of vertices. If 
    \begin{enumerate}[label=\rm(\alph*)]
        \item \label{itm:indep} $M_1\cup M_2\cup (V(D)-(\Out(M_1)\cup \In(M_1)\cup \In(M_2)))$ is an independent set, 
        \item \label{itm:out} $\Out(M_1)\subseteq Z$, 
        \item \label{itm:in} $Z\cap \In(M_1)= M_1$, and 
        \item \label{itm:side} $Z\cap \In(M_2)=\emptyset$, 
    \end{enumerate}
    then $Z$ is a closed set.
\end{lemma}
\begin{proof}
    Let $x\in Z$ and $y$ be an out-neighbor of $x$.
    We want to show that $y\in Z$.
    If $x\in \Out(M_1)$, then trivially $y\in \Out(M_1)\subseteq Z$. So we may assume that $x\notin \Out(M_1)$. 
    In particular, $x\notin M_1$ because $M_1\subseteq \Out(M_1)$.
    By \ref{itm:in}, $x\notin \In(M_1)$ and therefore $x$ is in $V(D)-(\In(M_1)\cup \Out(M_1))$.
    By~\ref{itm:side}, $x\notin \In(M_2)$.
    By \ref{itm:indep}, $y\in \Out(M_1)\cup \In(M_1)\cup \In(M_2)-(M_1\cup M_2)$.
    However, since $x\notin \In(M_2)$, we know that $y\notin \In(M_2)$.
    And as $x\notin \In(M_1)$, we deduce that $y\notin \In(M_1)$.
    Thus, $y\in \Out(M_1)$, implying that $y\in Z$ by \ref{itm:out}.
\end{proof}

\begin{lemma}\label{lem:dagcomp}
    Let $D$ be an $n$-vertex directed acyclic graph with no skew matching of size $\ell+1$.
    Then there is a list of at most $O(n^{\ell})$ pairs $(X_i,Y_i)$
    of disjoint subsets $X_i$, $Y_i$ of $V(D)$
    such that 
    \begin{enumerate}[label=\rm(\roman*)]
    \item  \label{item1} for any subset $U$ of $V(D)-(X_i\cup Y_i)$,
        $X_i\cup U$ is a closed set, 
    \item \label{item2} for every closed set $K$, there is $i$ such that 
        $K\cap (X_i\cup Y_i)=X_i$.
    \end{enumerate}
    Moreover, such a list can be found in time $O(n^\ell (n+\abs{E(D)}))$.
\end{lemma}
\begin{proof}
    There are at most $O(n^\ell)$ pairs $(M^i_1,M^i_2)$ of 
    disjoint sets $M^i_1$ and $M^i_2$ of vertices such that $\abs{M^i_1\cup M^i_2}\le \ell$, $\Out(M^i_1)\cap M^i_2=\emptyset$, and $M^i_1\cup M^i_2\cup (V(D)-(\Out(M^i_1)\cup \In(M^i_1)\cup \In(M^i_2)))$ is an independent set.
    For each of such pair $(M^i_1,M^i_2)$, let $X_i=\Out(M^i_1)$ and $Y_i=\In(M^i_1)\cup \In(M^i_2)-M^i_1$.
    Since $D$ is acyclic, $\Out(M^i_1)\cap \In(M^i_1)=M^i_1$. 
    As $\Out(M^i_1)\cap M^i_2=\emptyset$, we deduce that $X_i\cap Y_i=\emptyset$.
By \zcref{lem:decoding}, \ref{item1} holds
    and by applying \zcref{lem:encoding} with $X=K$ and $Y=\emptyset$, we deduce that \ref{item2} holds.
\end{proof}

\begin{proposition}\label{prop:digraphcomp}
    Let $D$ be an $n$-vertex directed graph with no skew matching of size $\ell+1$.
    Then there is a list of at most $O(n^{\ell})$ triples $(X_i,Y_i,\mathcal P_i)$
    of disjoint subsets $X_i$, $Y_i$ of $V(D)$
    and a partition $\mathcal P_i$ of $V(D)-(X_i\cup Y_i)$
    such that 
    \begin{enumerate}[label=\rm(\roman*)]
    \item  for any subset $\mathcal Q$ of $\mathcal P_i$,
        $X_i\cup \bigcup_{C\in \mathcal Q} C$ is a closed set, 
    \item  for every closed set $K$, there is $i$ such that 
        $K\cap (X_i\cup Y_i)=X_i$
        and 
        $C\cap K\in\{\emptyset,C\}$ for all $C\in \mathcal P_i$.
    \end{enumerate}
Moreover, this list can be found in time $O(n^{\ell+2})$.
\end{proposition}
\begin{proof}
    Let $D'$ be the directed acyclic graph obtained from~$D$ by contracting each strongly connected component into a single vertex. 
    We claim that $D'$ has no skew matching of size $\ell+1$.
    Suppose for a contradiction that $D'$ has a skew matching
    $(A_1,B_1),\ldots,(A_{\ell+1},B_{\ell+1})$, where each $A_i$, $B_i$ is a vertex of $D'$.
    By the definition of a contraction, each vertex of $D'$ corresponds to a strongly connected component of $D$,
    and an arc $A_i\to B_i$ in $D'$ means there is an arc from some vertex of the component $A_i$
    to some vertex of the component $B_i$ in $D$.
    Choose such vertices $a_i\in A_i$ and $b_i\in B_i$ with $a_i\to b_i$ in $D$.
    Since the skew matching in $D'$ uses pairwise distinct vertices,
    all components $A_i$, $B_i$ are distinct and thus the chosen vertices are pairwise distinct as well.
    Moreover, for $i<j$, the non-edge condition $(A_i,B_j)\notin E(D')$ implies that
    there is no arc from any vertex of $A_i$ to any vertex of $B_j$ in $D$,
    and the non-edge condition $(B_i,A_j)\notin E(D')$ implies that
    there is no arc from any vertex of $B_i$ to any vertex of $A_j$ in $D$.
    Hence $(a_1,b_1),\ldots,(a_{\ell+1},b_{\ell+1})$ is a skew matching in $D$,
    contradicting the assumption.
    Now, we deduce the conclusion by applying \zcref{lem:dagcomp}.
    If $(X_i',Y_i')$ is a member of the list obtained from \zcref{lem:dagcomp} for $D'$, 
    then we construct $(X_i,Y_i,\mathcal P_i)$ by expanding each vertex of $D'$ into a set of vertices in the strongly connected components
    and taking $\mathcal P_i$ be the partition of $V(D)-(X_i\cup Y_i)$ into strongly connected components.
\end{proof}

\section{Final proof}\label{sec:proof}

\begin{theorem}[label=thm:main,store=main]
    Let $f$ be a connectivity function on an $n$-element set $V$.
Then for every non-negative integer~$k$, 
    there exists
    a collection  
    $\{(X_i,Y_i,\mathcal P_i)\}_{i\in I}$
    of tuples
    of disjoint subsets $X_i$, $Y_i$ of $V$
    and a partition~$\mathcal P_i$ of $V-(X_i\cup Y_i)$ 
    for each~$i\in I$
    such that 
    $\abs{I}\le O(n^{4k})$ and 
    $\{X\subseteq V: f(X)=k\}$ 
    is precisely equal to 
    \[ \left\{ \left(X_i\cup \bigcup_{C\in \mathcal Q} C\right) :
    \mathcal Q\subseteq \mathcal P_i\text{ for some $i\in I$}
    \right\}.
    \] 
    Furthermore, such a collection can be constructed in time $O(n^{2k+7}\gamma+n^{2k+8}+n^{4k+2})$, where $\gamma$ is the time to compute $f(X)$ for any set $X$.
\end{theorem}
\begin{proof}
Let $f^*(S,T)=\min_{S\subseteq X\subseteq V-T}f(X)$.
    By \zcref{thm:orlin}, 
    $f^*$ can be computed in time $O(n^5\gamma+n^6)$.
    For each $X\subseteq V$ with $f(X)=k$, there is a pair $(S,T)$ of disjoint subsets such that 
    \[ f^*(S,T)=f(X)=k, ~S\subseteq X\subseteq V-T, \text{ and }\max(\abs{S},\abs{T})\le k\]
    by \zcref{lem:findbase}.
    Let $D$ be the blocking digraph for $(S,T)$ with respect to $f^*$.
    Let $D'=D-(\Out(S^\circ)\cup \In(T^\circ))$.
    Then $Y$ is a closed set containing $S^\circ$ and not containing $T^\circ$ if and only if 
    $Y=\Out(S^\circ)\cup Y'$ for a closed set $Y'$ of $D'$.

    By \zcref{lem:nomatching}, $D'$ has no skew matching of size $2k+1$.
    By \zcref{prop:digraphcomp}, we obtain list of at most $O(n^{2k})$ triples $(X_i,Y_i,\mathcal P_i)$
    of disjoint subsets $X_i$, $Y_i$ of $V(D')$
    and a partition $\mathcal P_i$ of $V(D')-(X_i\cup Y_i)$
    such that 
    \begin{enumerate}[label=\rm(\roman*)]
    \item   for any subset $\mathcal Q$ of $\mathcal P_i$,
        $X_i\cup \bigcup_{C\in \mathcal Q} C$ is a closed set of $D'$, 
    \item  for every closed set $K$ of $D'$, there is $i$ such that 
        $K\cap (X_i\cup Y_i)=X_i$
        and $K\cap C\in\{\emptyset,C\}$ for all $C\in\mathcal P_i$.
    \end{enumerate}
    Now we construct our list by adding all vertices of $S$ and $\Out(S^\circ)-\{S^\circ\}$ into $X_i$ and 
    all vertices of $T$ and $\In(T^\circ)-\{T^\circ\}$ into $Y_i$.
    There are at most $O(n^{2k})$ pairs $(S,T)$ of disjoint subsets $S, T\subseteq V$ of size at most~$k$, from which it directly follows that $|I| \leq O(n^{4k})$.

    In order to construct all the required tuples algorithmically, we iterate all $O(n^{2k})$ pairs $(S, T)$.
    For each, we construct the blocking digraph $D$ from the definition within $O(n^2)$ evaluations of~$f^*$ and then directly apply \zcref{prop:digraphcomp}.
    The time complexity bound follows.
\end{proof}

\section{Application to submodular minimum bisection}\label{sec:cor}

The \textsc{submodular minimum bisection} problem is for an input connectivity function~$f$ on an $n$-element set~$V$ with an integer~$k$ to find a partition of~$V$ into two sets $A$ and $B$ such that $\abs{A}=\lfloor n/2\rfloor$, $\abs{B}=\lceil n/2\rceil$, and $f(A)\le k$.

More generally, let us try to find a set $A$ with $\abs{A}\in T$ for some fixed set $T$, while having $f(A)\le k$. 
Furthermore, we can replace $\abs{A}$ with $\abs{A\cap W}$ for a fixed set $W$.
We can prove a stronger statement by replacing $f(A)\le k$ with $f(A)=k$, because we can run the algorithm $k+1$ times.
We remark that Pilipczuk and Sokołowski observed previously that their low-rank structure theorem \cite[Theorem 6.1]{BPPSS2025} implies such an algorithm for the specific case of cut-rank function of graphs, i.e., for the case where $f = \rho$.
\begin{corollary}[store=bisection]\label{cor:bisection}
    Let $f$ be a connectivity function on an $n$-element set~$V$.
    Let $W$ be a subset of $V$. 
In time $O(n^{2k+7}\gamma+n^{2k+8}+n^{4k+2})$, for a set $T$ of integers,
    we can find a set $A$ such that $\abs{A\cap W}\in T$
    and $f(A)=k$
    or confirm that there is no such set $A$,
    where $f$ is given by an oracle that answers $f(X)$ for any set $X$ in time $\gamma$.
\end{corollary}
\begin{proof}
    By \zcref{thm:main}, in time $O(n^{2k+7}\gamma+n^{2k+8}+n^{4k+2})$ we can construct
    the list $\{(X_i,Y_i,\mathcal P_i)\}_{i\in I}$ with $\abs{I}\le O(n^{4k})$.
It is enough to show that for each $i$, we can decide whether there exists a subset~$\mathcal Q$ of~$\mathcal P_i$ such that 
    \[
        \abs{\left(X_i\cup \bigcup_{U\in \mathcal Q} U\right)\cap W}\in T.
    \]
    Let $A_1,A_2,\ldots,A_\ell$ be the elements of~$\mathcal P_i$. 
    Let $s_0=\abs{X_i\cap W}$ and $s_j=\abs{A_j\cap W}$ for all $j\in[\ell]$.
    
    We present an algorithm based on dynamic programming. 
    For each $j\in\{0,1,\ldots,\ell\}$ and each integer $t\in\{0,1,\ldots,\abs{W}\}$,
    we maintain a value $I_j(t)\in 2^{[j]}\cup\{\bot\}$, where $I_j(t)\neq\bot$ means that
    \[
        s_0+\sum_{p\in I_j(t)} s_p=t.
    \]
    Initialize $I_0(s_0)=\emptyset$ and $I_0(t)=\bot$ for all $t\neq s_0$.
    For each $j\in[\ell]$ and each $t\in\{0,1,\ldots,\abs{W}\}$, update as follows:
    \begin{itemize}
        \item If $I_{j-1}(t)\neq\bot$, set $I_j(t)=I_{j-1}(t)$.
        \item Otherwise, if $t\ge s_j$ and $I_{j-1}(t-s_j)\neq\bot$, set
        $I_j(t)=I_{j-1}(t-s_j)\cup\{j\}$.
        \item Otherwise set $I_j(t)=\bot$.
    \end{itemize}
    Therefore, for this fixed $i$, a feasible set exists if and only if
    $I_\ell(t)\neq\bot$ for some $t\in T$.
    In that case, let $\mathcal Q=\{A_p:p\in I_\ell(t)\}$ and output
    \[
        A:=X_i\cup\bigcup_{U\in\mathcal Q}U.
    \]
    By construction, $\abs{A\cap W}=t\in T$, and by \zcref{thm:main}, $f(A)=k$.

    For fixed $i$, the dynamic programming has $O(\ell\abs{W})\le O(n^2)$ states and transitions.
    Hence it runs in $O(n^2)$ time.
    Repeating this for all $i$ adds $O(n^2\cdot \abs{I})=O(n^{4k+2})$ time.
    Therefore, the total running time is
    $O(n^{2k+7}\gamma+n^{2k+8}+n^{4k+2})$.
\end{proof}

Recall that for a graph $G$ and a subset $X$ of $V(G)$, if $\eta(X)$ is the connectivity function counting edges between $X$ and $V(G)-X$.
As discussed in \zcref{sec:intro}, $\eta$ is a connectivity function on $V(G)$.
The minimum bisection problem for $\eta$ was previously studied by 
Cygan, Lokshtanov, Pilipczuk, Pilipczuk, and Saurabh \cite{CyganLPPS19}, who showed that it is fixed-parameter tractable.

We would like to ask whether the minimum bisection problem for other connectivity functions are fixed-parameter tractable. More generally 
is it true that the submodular minimum bisection is fixed-parameter tractable 
for all connectivity functions?

\paragraph{Acknowledgements.}
Both authors would like to thank the organizers of the LOGALG 2025 workshop held at Technische Universit\"at Wien in November 2025.
The first-named author would like to thank Micha\l{} Pilipczuk for kindly explaining the sketch of the proof of \cite{BPPSS2025}.
\providecommand{\bysame}{\leavevmode\hbox to3em{\hrulefill}\thinspace}
\providecommand{\MR}{\relax\ifhmode\unskip\space\fi MR }
\providecommand{\MRhref}[2]{\href{http://www.ams.org/mathscinet-getitem?mr=#1}{#2}
}
\providecommand{\href}[2]{#2}

\bibliographystyle{amsplain}
\end{document}